[1994/12/01]
\documentclass[draft]{article}
\usepackage{url}
\title{A simple proof on the inequality of arithmetic and geometric means}
\author{Haoxiang Lin}

\usepackage{amsmath,amsthm}
\usepackage[left=1in, right=1in, top=1in, bottom=1in]{geometry}

\chardef\bslash=`\\ 

\theoremstyle{definition}

\theoremstyle{remark}

\newcommand{\thmref}[1]{Theorem~\ref{#1}}

\newcommand{\eval}[2][\right]{\relax
  \ifx#1\right\relax \left.\fi#2#1\rvert}

\theoremstyle{definition}
\newtheorem{theorem}{Theorem}[section]

\theoremstyle{definition}
\newtheorem{definition}[theorem]{Definition}

\theoremstyle{remark}

\numberwithin{equation}{section}

\begin{document}
\date{}
\maketitle

\begin{abstract}

In this short paper we show that the inequality of arithmetic and geometric means is reduced to another interesting inequality,
and a proof is provided.

\end{abstract}

\section{Introduction}

Given $n$ arbitrary real numbers $a_1, a_2, \cdots, a_n$, we define their arithmetic and geometric means as following:
\begin{definition}
The \textit{Arithmetic Mean} is:
\begin{equation}
AM(a_1, a_2, \cdots, a_n) = \frac{a_1 + a_2 + \cdots + a_n}{n}
\end{equation}
\end{definition}

\begin{definition}
If such $n$ real numbers are all non-negative, the \textit{Geometric Mean} is:
\begin{equation}
GM(a_1, a_2, \cdots, a_n) = \sqrt[\leftroot{-2}\uproot{2}n]{a_1 \cdot a_2 \cdots a_n}
\end{equation}
\end{definition}

The inequality of arithmetic and geometric means states that the arithmetic mean is greater than or equal to the geometric mean
if those real numbers are all \emph{positive}:
\begin{theorem}\label{agm}
For arbitrary $n$ positive real numbers $a_1, a_2, \cdots, a_n$, the inequality
\begin{equation}
\frac{a_1 + a_2 + \cdots + a_n}{n} \geq \sqrt[\leftroot{-2}\uproot{2}n]{a_1 \cdot a_2 \cdots a_n}
\end{equation}
holds, with equality if and only if $a_1 = a_2 = \cdots = a_n$
\end{theorem}

The inequality of arithmetic and geometric means is so famous that there are various proofs in the literature
~\cite{Wiki, Bencze98, Chong76, Miller03, Uchida08, Xia99}.
In this short paper we provide a simple proof which uses another interesting inequality.

\section{Proof}

If we use the following notations:
\begin{align*}
& x = \frac{a_1 + a_2 + \cdots + a_n}{n}\\
& d_i = a_i - x
\end{align*}

We see that $x > 0, x + d_i = a_i > 0, \sum_{i=1}^{n} d_i = 0$
and $a_1 = a_2 = \cdots = a_n$ implies that $d_1 = d_2 = \cdots = d_n = 0$.
Then \thmref{agm} is actually reduced to the following inequality:
\begin{theorem}\label{lin}
Let $x$ be a positive real number, $d_1, d_2, \cdots, d_n$ be real numbers that
each $ d_i > -x$ and $\sum_{i=1}^{n} d_i = 0$, then
\begin{equation}\label{lin1}
x \geq \sqrt[\leftroot{-2}\uproot{2}n]{(x + d_1) \cdot (x + d_2) \cdots (x + d_n)}
\end{equation}
\begin{equation}\label{lin2}
x^n \geq (x + d_1) \cdot (x + d_2) \cdots (x + d_n)
\end{equation}
Both equalities hold if and only if $d_1 = d_2 = \cdots = d_n = 0$.
\end{theorem}

\begin{proof}
Since $x$ and $x + d_i$ are positive real numbers, \eqref{lin1} is equivalent to \eqref{lin2}.
We prove the second inequality using induction on $n$.
\begin{enumerate}
\item When $n = 1$ and $2$, it is easy to verify the correctness.

\item Suppose that when $n = k(\geq 2)$, the inequality is true. That is
\begin{equation*}
x^k \geq (x + d_1) \cdot (x + d_2) \cdots (x + d_k)
\end{equation*}

\item Now assume $n = k+1$.
If $d_1 = d_2 = \cdots = d_{k+1} = 0$, the equality is trivially true.
Otherwise suppose that $d_1, d_2, \cdots, d_{k+1}$ are not all zero.
Then there must be one $d_u > 0$, one $d_v < 0$, and $u \neq v$ since $n > 2$.
Without loss of generality, assume that $d_{k+1} > 0$ and $d_k < 0$.
Since $d_k > -x$, it is obviously true that $d_k + d_{k+1} > -x$.
It is clear that $x, d_1, d_2, \cdots, d_{k-1}, (d_k + d_{k+1})$ also meet the prerequisites of the inequality, therefore
\begin{equation*}
x^k \geq (x + d_1) \cdot (x + d_2) \cdots (x + (d_k + d_{k+1}))
\end{equation*}
And we get
\begin{align*}
x^{k+1} & \geq (x + d_1) \cdot (x + d_2) \cdots (x + (d_k + d_{k+1})) \cdot x\\
& = (x + d_1) \cdot (x + d_2) \cdots (x^2 + (d_k + d_{k+1}) \cdot x)\\
& > (x + d_1) \cdot (x + d_2) \cdots (x^2 + (d_k + d_{k+1}) \cdot x + d_k \cdot d_{k+1}) \quad \text{since $d_k \cdot d_{k+1} < 0$}\\
& = (x + d_1) \cdot (x + d_2) \cdots (x + d_k) \cdot (x + d_{k+1})
\end{align*}
\end{enumerate}

Now we complete the proof, and the equality holds if and only if $d_1 = d_2 = \cdots = d_n = 0$.
\end{proof}

\end{document}